\documentclass[11pt]{article}
\usepackage[centertags]{amsmath}
\usepackage{amsfonts}
\usepackage{amssymb}
\usepackage{amsthm}
\usepackage{newlfont}
\usepackage[all,poly]{xy}

\newcommand{\comment}[1]{}

\newtheorem{theorem}{Theorem}
\newtheorem{lemma}{Lemma}
\newtheorem{rem}{Remark}

\newtheorem{proposition}{Proposition}

\newtheorem{definition}{Definition}

\newtheorem{corollary}{Corollary}
\begin{document}

\title{One-dimensional infinite memory imitation models with noise}
\author{ Emilio De Santis \\
{\small Dipartimento di Matematica }\\
{\small  Universit\`a di Roma La Sapienza}\\
{\small \texttt{desantis@mat.uniroma1.it}} \and Mauro Piccioni \\
{\small Dipartimento di Matematica }\\
{\small  Universit\`a di Roma La Sapienza}\\
{\small \texttt{piccioni@mat.uniroma1.it}}} \maketitle

\begin{abstract}
In this paper we study  stochastic process indexed by $\mathbb {Z}$
constructed from certain transition kernels depending on the whole
past. These kernels prescribe that, at any time, the current state
is selected by looking only at a previous random instant. We
characterize uniqueness in terms of simple concepts concerning
families of stochastic matrices, generalizing the results previously
obtained in De Santis and Piccioni (J. Stat. Phys.,
150(6):1017--1029, 2013).
\end{abstract}

\medskip

\medskip \noindent \textbf{keywords:} Perfect simulation, Coupling, Chains
with complete connections.


\section{Introduction and main definitions}

\comment{
$$\xymatrix{
\bullet \ar@/_/[rr] & \bullet & \bullet
\ar@/_0.5pc/[rr]  & \bullet \ar@/^1pc/[rrr] & \bullet
 \ar@/_1pc/[rr]  & \bullet & \bullet
\ar@/_1.5pc/[ll]
} $$

$$\xymatrix{
1\ar[r] \ar@/_/[rr] & 2\ar[r] & 3 \ar[r]
\ar@/_0.5pc/[rr]  & 4 \ar[r]  \ar@/^1pc/[rrr] & 5
\ar[r]  \ar@/_1pc/[rr]  & 6 \ar[r] &7 \ar[r]
\ar@/_1.5pc/[rr]  & 8 \ar[r] & 9
} $$
}

This paper is concerned with stochastic processes indexed by
$\mathbb{Z}$ taking values in a finite alphabet $G$, constructed
from a \emph{transition kernel} which depends on the whole past, i.e. a map
$p : G \times G^{-\mathbb{N}_+} \to [0,1]$ such that, for any choice
of $\mathbf {w}_{-\infty}^{-1}=(w_{-1},w_{-2},\ldots) \in
G^{-\mathbb{N}_+}$, $p(\cdot| \mathbf {w}_{-\infty}^{-1} )$ is a
probability measure on $G$. In the literature these models appear under various names, as
\emph{chains with complete connections} \cite{Iosifescu},
\emph{g-functions} \cite{Bramson} or \emph{processes with long memory} \cite{CFF}.

Two alternative ways to associate a $G$-valued stochastic process,
i.e. a probability measure on $G^{\mathbb{Z}}$, to a transition
kernel are possible. The first deals with processes with a  boundary
condition. Let $ \mathbf{w}= (w_n, n \in \mathbb{Z}) $ be an
arbitrary \emph{configuration}, i.e. an element of $G^{\mathbb{Z}}$,
possibly random, and let $r \in \mathbb{Z}$; we say that
${\mathbf{X}}^{r, \mathbf{w}}=(X^{r, \mathbf{w}}_n, n \in
\mathbb{Z})$ is \emph{governed by the kernel} $p$ \emph{with
boundary condition} $\mathbf {w}$ \emph{from the instant} $r$, if
$X_n^{r, \mathbf{w}}= w_n$ for $n \leq r$ and
\begin{equation}\label{bobob}
    P( X_n^{r, \mathbf{w}}= g |X_{n-1}^{r, \mathbf{w}} ,X_{n-2}^{r, \mathbf{w}}, \ldots  )=
    p( g |X_{n-1}^{r, \mathbf{w}} ,X_{n-2}^{r, \mathbf{w}},
    \ldots) \text{ a.s.},
\end{equation}
for any $n>r$. It is clear that, given  the law of $\mathbf{w}$ (in
particular if it is a deterministic sequence), the law of this
process is uniquely defined. If for some strictly decreasing
sequence $ (r_n )$, ${\mathbf{X}}^{r_n, \mathbf{w}}$ converges
weakly in the product topology of $G^{\mathbb{Z}}$, the limit is
said to be an \emph{infinite volume limit}. By weak compactness of
the set of probability measures on a compact space, this set is
non-empty: it reduces to a single element $\mu$ if and only if
${\mathbf{X}}^{r, \mathbf{w}}$ converges weakly to $\mu$, as $r \to
-\infty$, irrespectively of $\mathbf{w}$. Infinite volume limits are
mainly considered in the theory of multi-dimensional random fields
\cite{Georgii,Grimmett}. Here we find convenient to borrow the
standard usage in the multi-dimensional framework to introduce
boundary conditions from an entire configuration on $\mathbb{Z}$,
rather than shifting to the left a configuration defined only on the
half-line $-\mathbb{N}_+$.

The second construction, more often used in the one-dimensional \emph{time-directional} context
we are concerned with, is
to declare directly a process
${\mathbf{X}}=(X_n, n \in \mathbb{Z})$,  equivalently its law,
to be \emph{compatible} with $p$ if
\eqref{bobob} holds for any $n \in \mathbb{Z}$. This is analogous to the Dobrushin-Lanford-Ruelle
definition in the theory of multi-dimensional  random fields \cite{Georgii,Grimmett}.
Compatible laws are immediately seen to be infinite volume
limits; indeed if ${\mathbf{W}}$ is compatible with $p$ and we choose it as
a boundary condition, then ${\mathbf{X}}^{r, \mathbf{W}}$ has
the same law of ${\mathbf{W}}$, for any $r \in \mathbb{Z}$: here we profit
of having allowed random boundary conditions.
Conversely,
since \eqref{bobob} is equivalent to
$$
E [ \mathbf{1}_{\{g\}} ({X}^{r, \mathbf{w}}_n ) h({X}^{r, \mathbf{w}}_{n-1} ,
 \ldots ,  {X}^{r, \mathbf{w}}_{n-m}  ) ] =
 E [  p( g |X_{n-1}^{r, \mathbf{w}} ,X_{n-2}^{r, \mathbf{w}},
    \ldots)  h({X}^{r, \mathbf{w}}_{n-1} ,
 \ldots ,  {X}^{r, \mathbf{w}}_{n-m}  ) ] ,
$$
for any positive integer $m$ and any real function $h$ defined on
$G^m$, this relation is maintained in the limit provided
$p(g|\cdot)$ is continuous for any $g \in G$. In this paper only
continuous kernels will be considered, therefore we will identify
infinite volume limits with compatible laws, denoting their set with
$\mathcal{G}(p)$. In the proofs both characterizations of
$\mathcal{G}(p)$ will be found useful.

Notice that elements of $\mathcal{G}(p)$ are not necessarily
stationary, i.e. translation invariant, but from a non-stationary
element of $\mathcal{G}(p)$ one can produce a stationary one by
performing Cesaro averages of shifts over a finite window increasing
to $\mathbb{Z}$. Thus if $\mathcal{G}(p)$ reduces to a single
element, it has to be stationary. On the contrary, it is possible
that  $|\mathcal{G}(p)|>1$ but this set contains only one stationary
element; indeed we will present later a situation in which this
happens. Notice that, $\mathcal{G}(p)$ being convex, in case of non
uniqueness $\mathcal{G}(p)$ has infinitely many elements.

Uniqueness conditions for general transition kernels of the form
\eqref{bobob} are scattered in the literature for various decades.
Some of these results refer to a dynamical systems setting,  see
e.g. \cite{Keane, Walters, Johansson}. The use of techniques of a
more probabilistic flavor, in particular \emph{coupling techniques},
has increased in time, see e.g. \cite{Lalley, Berbee, Stenflo}. The
work \cite{CFF} has started a constructive approach, focused to the
design of \emph{perfect simulation schemes} for the unique
compatible measure. In a number of cases this has allowed to prove
not only the uniqueness, but also the existence of a compatible law,
when $G$ is countable. Finally, multi-dimensional statistical
mechanics techniques, such as the  Dobrushin criterion, have
recently been used also in this setting \cite{Fernandez04, FM05}.
For perfect simulation in the multi-dimensional case the reader is
addressed to e.g. \cite{DL, DPExact, GLO}; also the continuity assumption can be relaxed,
as in \cite{DM}.

The various sufficient conditions for uniqueness usually take a
suitable positivity condition together with some regularity
assumption on the kernel $p$. The latter allows to control the
behavior of the range of the functions $p(g|\mathbf
{w}^{-1}_{-r}\cdot)$ on $G^{-\mathbb {N}_{+}}$, for fixed $g \in G$
and $\mathbf {w}^{-1}_{-\infty} \in G^{-\mathbb {N}_+}$, as $r$ gets
large. Regularity assumptions of some sort are actually needed for
uniqueness, as shown in \cite{Bramson}, where an example of a
positive transition kernel has been given with
 a strong ``dependence on the remote past" that gives rise to different infinite volume limits.

In order to motivate the class of kernels considered in the paper it
is useful to recall the setting of \cite{CFF}. In this paper they write down
a decomposition of a continuous kernel of the following form:
\begin{equation}\label{cffkern}
p(g |\mathbf{w}_{-\infty}^{-1})=\theta_0 \nu (g)+\sum_{k=1}^{\infty} \theta_k
P_{(k)}(g; w_{-1},....w_{-k}),
\end{equation}
where $\nu$ is a probability distribution on $G$, $\theta =(\theta_n, n \in \mathbb{N})$ is
a probability distribution on the integers and for any $k \in \mathbb{N}_+$
$P_{(k)} : G \times G^k \to [0,1]$ is a transition kernel
depending only on the $k$-th most recent values. If $\theta_0>0$ and
$\theta_n$ decays to zero fast enough they provide a perfect
simulation algorithm  for the unique compatible measure.
The first assumption corresponds to positivity of $p(g|\cdot)$, for some $g \in G$,
whereas the second amounts again to a regularity
assumption on the kernel $p$.

The mixture decomposition presented in \cite{CFF} is not unique.
Other decompositions have been proposed to prove uniqueness
\cite{DPBackward, GG13, SandroG}, leading to relax not only the
regularity but also the positivity assumption in \cite{CFF}.

In the present paper we consider general transition
kernels of the following form
\begin{equation}\label{kern}
p(g |\mathbf{w}_{-\infty}^{-1})=\sum_{k\in \mathcal {A}} \theta_k
P_{(k)}(w_{-k},g),
\end{equation}
for some probability distribution $\theta$ supported by $\mathcal {A}\subset \mathbb {N}_+$ and
$P_{(k)}$ is a stochastic matrix on $G$,  for any $k \in \mathcal {A}$.
Since $\sum_{k\in \mathcal {A}}\theta_k=1$, any kernel of the form
\eqref{kern} is clearly continuous. A transition kernel of the above
form will be called an \emph{imitation kernel}.

 When the
$P_{(k)}$'s have the property that each row contains only a single
positive entry, necessarily equal to $1$ (as happens in particular
for permutation matrices), the updating rule \eqref{bobob} means
$X_n^{r, \mathbf{w}}=f_k(X_{n-k}^{r, \mathbf{w}})$ with probability
$\theta_k$, where $f_k$ is a function on $G$ obtained from
$P_{(k)}$. We refer to these cases as imitation kernels
\emph{without noise}. Otherwise we speak about imitation kernels
\emph{with noise}. Imitation kernels without noise are in some
sense, to be clarified later, the most interesting to consider.

For general kernels of the form \eqref{kern} the value $X_{r+1}^{r,
\mathbf{w}}$ can be drawn in the following way. An integer $K_{r+1}$
is chosen at random according to the distribution $\theta$, and the
value of the boundary condition $w_{r+1-K_{r+1}}$ is read. Then
$X_{r+1}^{r, \mathbf{w}}$ is drawn from the $w_{r+1-K_{r+1}}$-th row
$P_{(K_{r+1})}(w_{r+1-K_{r+1}},\cdot)$ of the matrix
$P_{(K_{r+1})}$. To perform this step, it is convenient to make
reference to a sequence $(f_{(k)}, k \in \mathbb{N}_+)$ of
\emph{coupling functions} $f_{(k)}:G \times [0,1] \to G$, having the
property that, for any $k \in \mathbb{N}_+$, $f_{(k)}(g,U)$ is
distributed as $P_{(k)}(g,\cdot)$ whenever $U$ is a random variable
uniformly distributed in $[0,1]$, for $g \in G$. So, if $U_{r+1}$ is
uniformly distributed in $[0,1]$,
$f_{(K_{r+1})}(w_{r+1-K_{r+1}},U_{r+1})$ yields $X_{r+1}^{r,
\mathbf{w}}$. This updating rule can be iterated to produce the
values of the process ${\mathbf{X}}^{r, \mathbf{w}}$ at all sites
$n>r$ by drawing a random sample $(K_n, n
> r)$ from $\theta$ and an independent random sample $(U_n, n > r)$
from the uniform distribution on $[0,1]$. It is clear that for imitation kernels without noise,
the $U_n$'s are not needed for the
construction of ${\mathbf{X}}^{r, \mathbf{w}}$.

Rather than proceeding \emph{forward} from the boundary sites, one
can proceed \emph{backwards} from any site of interest. In this case to produce the random variable $X_n^{r,
\mathbf{w}}$, with $n> r$, we have to follow the random walk
$\mathcal {T}^{(n)}=(T_k^{(n)}, k \in \mathbb{N})$
\begin{equation}\label{rwalk}
T_{k+1}^{(n)}=T_{k}^{(n)}-K_{T_{k}^{(n)}}, \text{ for } k \in
\mathbb{N},\,\,\,
\end{equation}
with $T_0^{(n)}=n$, whose distribution of decrements is $\theta$.

Let us define
\begin{equation}\label{indiceM}
M_r^{(n)} = \inf \{ k :  T_k^{(n)} \leq r \}, \hbox{ }
V_r^{(n)}=T_{M_r^{(n)}}^{(n)}
\end{equation}
which are the minimum number of steps leading the random walk to
land on a site below the \emph{threshold} $r$ and the landing site,
respectively. The information on the boundary condition $
\mathbf{w}$ is propagated forward by applying recursively the
coupling functions in the following way
\begin{equation}\label{back}
X^{r,\mathbf{w}}_{T_{k-1}^{(n)}}=f_{(T_{k-1}^{(n)}-T_{k}^{(n)})}{(X^{r,
\mathbf{w}}_{T_{k}^{(n)}},U_{T_{k}^{(n)}})}, \,\,\,\,
k=M_r^{(n)},\ldots,1
\end{equation}
starting from $X^{r,\mathbf{w}}_{T_{M_r^{(n)}}}=w_{V_r^{(n)}}$.
Thus, at the end of the recursion one has, for any
$n >r $
\begin{equation}\label{F}
X_{n}^{r,\mathbf{w}}=F_{r,n}(       K_{m},U_{m}, r<m\leq n ;
w_{V_r^{(n)}})
\end{equation}
for some suitably defined  function $F_{r,n}$.

One can appreciate here that, if an additional zero order term
$\theta_0\nu(g)$, with $\theta_0>0$ and $\nu$ probability measure on
$G$, appears in the kernel \eqref{kern}, the $K_n$'s can also assume
the value $0$ with probability $\theta_0$. When this happens, one
stops the random walk from going further in the past and reads
directly the value at that site by sampling from $\nu$. Since this
event will  happen a.s., uniqueness always holds in this case.
Incidentally, $\theta_0 >0$ means $\sum_{k=1}^{\infty}\theta_k<1$,
the Dobrushin sufficient criterion for uniqueness for this kind of
kernels. It is not difficult to realize that a zero order term
\emph{cannot be singled out} when each of the columns of $P_{(k)}$
has a zero entry, for all $k \in \mathcal {A}$. In particular, this
happens for imitation kernels without noise, except in the trivial
case of some $P_{(k)}$ with all the rows equal to the same unit
versor.

Whenever for some pair of distinct sites $m,n \in \mathbb{Z}$, it
happens that $T_h^{(m)}=T_k^{(n)}$, for some positive integers $h$
and $k$, we say that the two random walks started from the sites $m$
and $n$ \emph{coalesce}. If this is the rightmost site in which this
happens, we say that $T_h^{(m)}=T_k^{(n)}$ is the \emph{coalescence
point} of the two random walks. In this case one has
$T_{h+l}^{(m)}=T_{k+l}^{(n)}$, for any $l \in \mathbb{N}$, hence for
$r\leq T_h^{(m)}$, it is $V_r^{(m)}=V_r^{(n)}$. As a consequence the
values $w_{V_r^{(m)}}$ and $w_{V_r^{(n)}}$ coincide, conveying all
the information about the boundary condition $\mathbf {w}$ needed to
compute both $X_{m}^{r, \mathbf{w}}$ and $X_{n}^{r,\mathbf{w}}$, by
means of the functions $F_{r,m}$ and $F_{r,n}$ defined in \eqref{F}.

If the random walks $\mathcal {T}^{(n)}$, started from $n \in
\Lambda$, with $\Lambda$ arbitrary finite subset of $\mathbb{Z}$,
coalesce a.s. we say that the distribution $\theta =(\theta_n, n \in
\mathcal {A})$ is \emph{coalescent}. In order to verify this
property it is enough to check it for a window $\Lambda$ made by two
adjacent sites of $\mathbb{Z}$. If we let two particles perform two
independent random walks started from these two sites, with the rule
that it is always the rightmost that moves, the distance between the
two particles is a Markov chain on the integers, the so called von
Schelling process \cite{Feller}, up to coalescence. This is again a
random walk with decrements following the law $\mathbf {\theta}$,
but with a reflection around the origin once the negative half-line
is hit. Thus coalescence of $\theta$ means that from any $n \in
\mathbb{N}_+$ the return of this process to the origin is almost
sure; this requires both some algebraic property for $\mathcal{A}$
and a control of the tail behavior of the $\theta_n$'s, see \cite{PW11}.

Back to imitation kernels, we mention that already in \cite{CFF} a
particular class of binary kernels of the form
\eqref{kern} was examined, in which $P_{(k)}$ took only the two possible
values
\begin{equation}\label{dimenticato}
I_2=\left(%
\begin{array}{cc}
  1 & 0 \\
  0 & 1 \\
\end{array}%
\right), J_2=\left(%
\begin{array}{cc}
  0 & 1 \\
  1 & 0 \\
\end{array}%
\right),
\end{equation}
for $k\in \mathcal {A}$, the so-called \emph{binary autoregressive kernels}.
However the presence of a zero order term $\theta_0 \nu (g)$ made the uniqueness problem trivial.
In our previous work \cite{DPautor}
we have considered binary autoregressive kernels with $\theta_0=0$,
equivalently with the $P_{(k)}$'s equal to either $I_2$ or $J_2$,
making a first step towards understanding the implications of the lack of positivity
for imitation kernels. The main result of that paper is that for a coalescent $\theta$ uniqueness holds.

In the present paper the results are completely general, and not restricted to the binary case.
The main result is that uniqueness for imitation kernels
can be characterized completely in terms of the properties of
what we call the \emph{$G$-stochastic function} induced by
the imitation kernel \eqref{kern}, namely the mapping
$$k \in \mathcal {A}\subset \mathbb {N}_+ \mapsto P_{(k)}.$$
The necessary and sufficient conditions generalize the well known
concepts of \emph{irreducibility} and \emph{aperiodicity} for a
single stochastic matrix.

Irreducibility is discussed in Section \ref{sec2}. Since the
presence of two irreducible classes implies the existence of two
different compatible laws (Proposition \ref{addnonuniq}) and states
not belonging to an irreducible class cannot appear in the support
of a compatible law (Proposition \ref{essentirred}), we are allowed
to focus our further study to irreducible kernels.

Aperiodicity is the subject of Section \ref{sec3}. Here a difference
with the case of a single stochastic matrix appears: the states are
constrained to have a period which is a multiple of the gcd
$d(\mathcal{A})$ of $\mathcal {A}$. But since any element of
$\mathcal {G}(p)$ has independent marginals along the residual
classes mod $d(\mathcal{A})$, the uniqueness problem is reduced to
any of them, for which  with an obvious rescaling
$d(\mathcal{A})=1$. This is the content of Proposition
\ref{sottografo}, which allows to correct a mistake occurred in
\cite{DPautor}. Furthermore, as it happens for finite Markov chains,
the presence of several periodic classes implies the existence of
different \emph{non-stationary} elements of $\mathcal {G}(p)$,
obtained one from the other by shifts (Theorem \ref{t2}). Notice
that this kind of \emph{phase transition} is entirely different from
the one in \cite{Bramson} that  concerns a positive kernel. At the
end of the section we prove an important lemma relating the
stationary elements of $\mathcal {G}(p)$ with the invariant
distributions of the stochastic matrix $\hat P=\sum_{k \in \mathcal
{A}} \theta_k P_{(k)}$.

In Section \ref{princris} we prove our main result (Theorem
\ref{t1}), which is analogous to the ergodicity theorem for finite
Markov chains: uniqueness holds for irreducible and aperiodic
imitation kernels. The unique invariant distribution $\hat
{\lambda}$ of $\hat P$ is identified as the single-site marginal of
the unique compatible law.

The fact that the tail behavior of the $\theta_k$'s does not enter
in the uniqueness conditions entails that, by keeping $\mathcal{A}$
fixed, but distributing enough mass to larger values of $k \in
\mathcal{A}$, we can construct examples of uniqueness in which any
of the general sufficient conditions appeared in the literature
fails. The uniqueness theorem appearing in \cite{DPautor}  is found
as a particular case, without assuming coalescence (Theorem
\ref{vecchionuovo}).

In Section \ref{perfect} we propose two simulation algorithms of the
CFTP type \cite{PW}, to construct the unique compatible law on any
finite window of $\mathbb {Z}$. The first, presented in Theorem
\ref{epsilone}, works when the distribution $\mathbf {\theta}$ is
known to be coalescent. When $\theta $ is not coalescent or at least
this is unknown, a threshold has to be specified, introducing an
error in the algorithm. In Theorem \ref{end} we prove that the error
introduced in this way can be made arbitrarily small pushing the
threshold towards $-\infty$. For this reason we call it an
$\varepsilon$-perfect simulation algorithm. In a situation of non
uniqueness, the algorithms presented here can still be used to
construct any stationary element of $\mathcal {G}(p)$. In the
irreducible but non-aperiodic case, it can also be proved that there
is only one stationary element.

Finally, in Section \ref{basta} we present a result for the case of
countable $G$. A sufficient condition for the existence of a unique
element in $\mathcal{G}(p)$, together with a perfect simulation
algorithm is obtained. For finite $G$, such a condition reduces to
irreducibility and aperiodicity. The  algorithm eliminates the
approximation error for non-coalescent $\theta$ but it can be
considerably more complicate for $G$ large or infinite.

\medskip

\section{Irreducibility of $G$-stochastic functions and uniqueness}
\label{sec2}

We start with a brief discussion of mappings defined on some $\mathcal {A} \subset \mathbb{N}_+$
with values in the set of stochastic matrices over the set $G$. We
call a mapping of this type a $G$-\emph{stochastic function}.

Recall that the free semigroup generated by $\cal {A}$ is the set
$\cal {A}^*=\bigcup_{n \in \mathbb{N}^+} \cal {A}^n$ of finite
$n$-tuples with elements in $\mathcal {A}$, for all positive
integers $n$, called \emph{words} in the sequel. It is indeed a semigroup
under concatenation, defined for $\mathbf{a}=(a_1,\ldots,a_n)$ and
$\mathbf{b}=(b_1,\ldots,b_m)$ by
$\mathbf{a}\mathbf{b}=(a_1,\ldots,a_n,b_1,\ldots,b_m)$. A
$G$-stochastic function defined on $\mathcal {A}$ extends to a
homomorphism of the semigroup $\cal {A}^*$ into the semigroup of
stochastic matrices on $G$ by associating to each $\mathbf
{a}=(a_1,\ldots,a_n)$ the stochastic matrix

\begin{equation}\label{conca}
P_\mathbf {a}=P_{(a_n)}\cdots P_{(a_1)}.
\end{equation}

Likewise, for any $\mathbf{a}=(a_1,\ldots,a_n)\in \cal {A}^*$, we
define the composition of coupling functions $f_{\mathbf{a}}: G
\times [0,1]^n\to G$ as
$$
f_{\mathbf{a}}(g;u_1,\ldots,u_n)=f_{(a_1)}(f_{(a_2)}(\ldots
f_{(a_{n-1})}(f_{(a_n)}(g;u_n),u_{n-1}),\ldots,u_{2}),u_1),
$$
for $g \in G$ and $u_i\in [0,1], i=1,\ldots,n$. We define the
\emph{depth} of $\mathbf{a}$ as $s (\mathbf{a})= \sum_{ i=1 }^n a_i
$.

Let us define a directed graph  $\Gamma_{\theta}$ with the sites of
$\mathbb{Z}$ as vertices, and  arcs joining  $n \in \mathbb{Z}$ with
$n-k$, whenever $k \in \cal{A}$. It is natural to visualize the
elements $\mathbf {a} \in \mathcal {A}^*$ as  paths of the graph
$\Gamma_{\theta}$. Once we have weighted  the  arc $(n,n-k )$ with
the probability $\theta_k>0$, we can assign a probability to any
path, given by the product of the probabilities of the arcs
belonging to the path. Now, for $\mathbf {a}=(a_1,\ldots,a_m) \in
\mathcal {A}^*$, $P_{\mathbf {a}}$ is the stochastic matrix used to
compute the value of $X_n^{r, \mathbf{w}}$ from the value
$X_{n-s(\mathbf {a})}^{r, \mathbf{w}}$, for $n-s(\mathbf {a})>r $,
whenever $K_n=a_1, K_{n-a_1}=a_2,\ldots,
K_{n-a_1-\ldots-a_{m-1}}=a_m$: an event which has probability
$\theta_{\mathbf {a}}=\theta_{a_1}\cdot \ldots \cdot \theta_{a_m}$.
Notice that infinitely many paths are associated to each $\mathbf
{a} \in \mathcal {A}^*$, differing in the starting site in
$\mathbb{Z}$. Also observe that the \emph{depth} $s(\mathbf {a})$ is
the distance of the last site of the path from the first one (see
Fig. 1).

The sample $\mathbf {K}=(K_n, n \in \mathbb {Z})$ selects a particular random subgraph
$\Gamma^{\mathbf{K}}$ of $\Gamma_{\theta}$, made of the arcs
$(m,m-K_m)$, for all $m \in \mathbb{Z}$. Likewise, the random walks $\mathcal {T}^{(n)}=
(T^{(n)}_k, k \in \mathbb {N})$, for $n \in \mathbb {Z}$, can be
actually seen as random walks on the graph $\Gamma_{\theta}$.

\medskip

$$\xymatrix{
\bullet & \bullet & \bullet & \bullet \ar@/_1.5pc/[lll]
& \bullet & \bullet & \bullet
& \bullet \ar@/_1.5pc/[llll] & \bullet & \bullet
\ar@/_1.5pc/[ll]
} $$

\begin{center}\label{fig1}
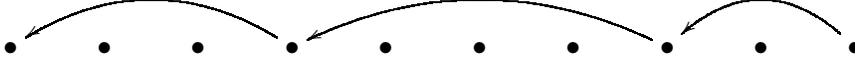
\small{Fig. 1. The word $\mathbf {a}=(2,4,3)$ and the corresponding path.}
\end{center}

\medskip

We can extend to a $G$-stochastic function $P_{(\cdot)}$ a number of
concepts which are well known for the ``standard" case of a single
stochastic matrix $P$, i.e. the $G$-stochastic function defined on
$\mathcal {A}=\{1\}$, with $P_{(1)}=P$.

\begin{definition}\label{irreducape}
Let $P_{(\cdot)}$ be a $G$-stochastic function defined on $\mathcal
{A}$ and $i \neq j \in G$. We say that $i \in G$
$P_{(\cdot)}$-communicates with $j\in G $ if there exists
$\mathbf{a} \in\mathcal {A}^*$ such that $P_{\mathbf{a} }(i,j)>0$.
We say that $i$ and $j$ $P_{(\cdot)}$-intercommunicate when $i$
communicates with $j$ and vice versa.
\end{definition}
As usual for the standard case, we declare that each $i \in G$
intercommunicates with itself, so that intercommunication becomes an
equivalence relation and $G$ is partitioned in intercommunicating
classes.

\begin{definition}\label{intercomun}An intercommunicating class $C$ is closed when $i \in C$
communicates with $j\in G$ implies $j \in C$ (therefore $j$
communicates with $i$). We say that a $G$-stochastic function
$P_{(\cdot)}$ is irreducible if for any $i$ and $j \in G$ there
exists $\mathbf {a} \in \mathcal {A}^*$ such that $P_a(i,j)>0$, that
is the whole $G$ is the only intercommunicating class.
More generally, $P_{(\cdot)}$ is essentially irreducible
if there exists a single intercommunicating class.
\end{definition}

It is easily verified that $P_{(\cdot)}$ is irreducible if and only
if $\hat P$ is irreducible.
As in the standard case it is possible to
decompose $G$ in a rather familiar way.

\begin{proposition}\label{decompose}
Let $P_{(\cdot)}$ be a $G$-stochastic function. There exists a
unique partition $\{R_1, \ldots , R_s, T\} $ of $G$, with $s \geq 1$, where
\begin{itemize}
   \item [1)] $R_h$ is a closed
   intercommunicating class, thus $\{P_{(k)}\vert_{R_h}, k \in \mathcal {A}\}$
   is irreducible, for $h=1,\ldots,s$;
   \item [2)] for any $i \in T$ there exists
   $j \in R_1 \cup\ldots \cup R_s$ and $\mathbf{a}
\in\mathcal {A}^*$ such that $P_{\mathbf{a} }(i,j)>0$.
\end{itemize}

\end{proposition}

\begin{proof} The proof is completely analogous to that for a single
finite stochastic matrix. In that case the sets $R_h$, $h=1,\ldots,s$ represent
the recurrent states, whereas the set $T$, which is the
union of the intercommunicating classes that are not closed,
represents the remaining transient states.
\end{proof}

The reader will notice that the above decomposition coincides with that concerning any
convex combination of the stochastic matrices $P_{(k)}, k \in \mathcal {A}$ with
positive weights, in particular
the matrix $\hat {P}$.
The following result allows to rule out a trivial case of non
uniqueness for kernels of the form \eqref{kern}.

\begin{proposition}\label{addnonuniq}
If the $G$-stochastic function $P_{(\cdot)}$ has more than one
closed intercommunicating class then for the
kernel \eqref{kern} one has $|\mathcal {G} (p)|>1$.
\end{proposition}
\begin{proof} The restriction of a kernel \eqref{kern} to a closed intercommunicating class $R_h$
is by itself a transition kernel $p_{R_h}$ on $R_h$. Identifying
$\mathcal {G}(p_{R_h})$ with a subset of $\mathcal {G}(p)$ in the
natural way, and taking into account that $\mathcal {G}(p_{R_h})$
are non empty and disjoint, for $h=1,\ldots,s$, the statement of the
theorem is immediately obtained.
\end{proof}

The simplest example of this sort is contained in \cite{DPautor}: if
$P_{(k)}=I_2$ for all $k \in \mathcal {A}$, then the two states are
two closed classes and the Dirac measures on the two constant
sequences are two elements of $\mathcal {G} (p)$.

The following proposition ensures that, when there is only one communicating class,
we can restrict the kernel to it.

\begin{proposition}\label{essentirred}
For a transition kernel of the form \eqref{kern}, consider the
corresponding $G$-stochastic function $P_{(\cdot)}$ and suppose that
$R$ is the union of all closed intercommunicating classes. Let $p_R$ be
the restriction of $p$ to $R$. Then $\mathcal {G}(p)=\mathcal
{G}(p_R)$ for the kernel \eqref{kern}.
\end{proposition}

\begin{proof}
The main step is to prove that if $\mathbf{X} \sim \mu \in
\mathcal{G}(p)$, then, for any $n \in \mathbb {Z}$, $P(X_n\in
{R}^c)=0$. For this it is enough to prove that
\begin{equation}\label{ultima}
\lim_{r \to -\infty}P(X_n^{r, \mathbf{w}}\in {R}^c)=0
\end{equation}
for any $\mathbf{w} \in G^{\mathbb {Z}}$.

Starting with $\mathcal {Z}_{0}=R^c$, we define recursively a
sequence of subsets $\mathcal {Z}_h \subset {R}^c$, for
$h=1,\ldots$, with strictly decreasing cardinality, until for some
integer $L$ it is $\mathcal {Z}_L= \emptyset$. During this
construction we will define $\mathbf {a}_h \in \mathcal {A}^*$ of
length $n_h$ and Borel sets $\Delta_{h} \subset [0,1]^{n_h}$ of
positive Lebesgue measure $\hbox{Leb}_{n_h}$, for
$h=0,1,\ldots,L-1$. The recursive construction is given by
$$
\mathcal {Z}_{h+1}=\{f_{\mathbf {a}_h}(i,\Delta_h), i \in \mathcal {Z}_{h}\}\cap  {R}^c
$$
and has the properties
\begin{enumerate}
    \item $\{f_{\mathbf {a}_h}(i,\Delta_h), i \in \mathcal {Z}_{h}\}\cap
    {R}\neq \emptyset$, $h=0,\ldots,L-1$
    \item For any $i\in \mathcal {Z}_h$, $|f_{\mathbf
    {a}_h}(i,\Delta_h)|=1$.

\end{enumerate}

Let us explain the generic step $h$ of the construction. Fix an
arbitrary state $j \in \mathcal {Z}_{h}$. By the definition of ${R}$
there exists a word $\mathbf {a}_h \in \mathcal {A}^*$ of length
$n_h$ such that $P_{\mathbf {a}_h}(j,{R})>0$. This ensures the
existence of $\Delta_h^*$ with a positive $n_h$-dimensional Lebesgue
measure such that $f_{\mathbf {a}_h}(j,\Delta_h^*)\in {R}$, hence
Property 1 is guaranteed. To obtain Property 2 one may need to
reduce $\Delta_h^*$ to some smaller $\Delta_h \subset \Delta_h^*$
keeping positive Lebesgue measure, which is clearly always possible
by finiteness of $\mathcal {Z}_{h}$.

Now we are in a position to prove \eqref{ultima}. Define the word
$\mathbf {a}$ of length $m=n_0+n_1+\ldots+n_{L-1}$ by the
concatenation $\mathbf {a}=\mathbf {a}_{L-1} \ldots \mathbf {a}_{0}$
and the Borel set $\Delta=\Delta_{L-1} \times \cdots \times
\Delta_{0} \subset [0,1]^m$. Let $c=\theta_{\mathbf {a}}\cdot
\hbox{Leb}_{m}(\Delta)>0$.

Recall the recursive construction of $X_n^{r, \mathbf{w}}$ in terms
of the random walk $\mathcal T^{(n)}=\{T_k^{(n)}, k \in \mathbb
{N}\}$ with the corresponding sequences $\{K_k=K_{T_k^{(n)}}, k \in
\mathbb {N}\}$, obtained through the relation \eqref{rwalk}, and let
$\{U_k=U_{T_k^{(n)}}, k \in \mathbb {N}\}$. If, for some integer
$l$, a segment $(K_l,\ldots,K_{l+m},U_l,\ldots,U_{l+m})$ belongs to
$\mathbf {a} \times \Delta$ with $T_{l+m}^{(n)}>r$, then $X_n^{r,
\mathbf{w}} \in {R}$, irrespectively of $\mathbf{w}$. Since a
segment of this kind will eventually occur with probability $1$,
\eqref{ultima} holds.

As a consequence, if $\mu \in \mathcal{G}(p)$ then $\mu({R}^{\mathbb
{Z}})=1$. Moreover, being $\mu $ compatible   and supported by
${R}^{\mathbb {Z}}$, it is actually in $\mathcal{G}(p_{{R}})$.
\end{proof}

\section{Periodicity of $G$-stochastic functions and
uniqueness}\label{sec3}

In the previous section we have justified to restrict our attention
to irreducible $G$-stochastic functions $P_{(\cdot)}$. In this
section we turn our attention to the notion of periodicity of a
state $i \in G$, which is slightly more delicate than in the
standard case. We start by observing that the depth $s$ is a
homomorphism of the free semigroup $\mathcal {A}^*$ into the
additive semigroup of positive integers $\mathbb {N}_+$. As a
consequence $s(\mathcal {A}^*)$ is a sub-semigroup of $\mathbb
{N}_+$. The period of $\mathcal {A}$ is defined as
\begin{equation}\label{periodoa}
    d(\mathcal {A})=gcd\{ s(\mathcal {A}) \}=gcd\{s(\mathcal {A}^*)\}.
\end{equation}
If $1 \in \mathcal{A}$, as it happens in the standard case, then
$s(\mathcal {A}^*)$ coincides with $\mathbb {N}_+$, and $d(\mathcal
{A})=1$.  Recall that, except for a finite number of elements, an additive semigroup of positive integers has
always the form $\{n_0d, (n_0+1)d,\dots\}$, where $d$ is the gcd of
the semigroup and $n_0$ is a suitable positive integer.

Now suppose that the transition kernel $p$ in \eqref{kern} has
$d(\mathcal {A})>1$. Then it is natural to move from $p$ to
\begin{equation}\label{redker}
   \bar {p}(g |\bar {\mathbf{w}}^{-1}_{-\infty}  )=\sum_{l \in \bar
{\mathcal {A}}} {\theta}_{ld(\mathcal {A})} {P}_{({ld(\mathcal
{A})})}(\bar {w}_{-l},g),
\end{equation}
with $\bar {\mathcal {A}}=\{l:ld(\mathcal {A}) \in \mathcal {A}\}$,
so that $ d(\mathcal {\bar A}) =1$. The following proposition allows
us to set  $ d(\mathcal {A}) =1$ in all the uniqueness proofs of the
present paper, without restriction of generality.

\begin{proposition}\label{sottografo}
For a transition kernel $p$ of the form \eqref{kern} let $d(\mathcal
{A})>1$ and define $\bar p$ as in \eqref{redker}. Then uniqueness
holds for $p$ if and only if it holds for $\bar p$.
\end{proposition}

\begin{proof}
    Let us consider the process $\mathbf {X}^{r, \mathbf{w}}$
    governed by the transition kernel $p$. Then for any
    $h=0,1,\ldots, d(\mathcal {A})-1$ the $d(\mathcal {A})$-marginal
    process $\bar {\mathbf{X}}_{(h)}$, defined by
    \begin{equation}\label{marginal}
    \bar {X}_{(h)k}= X^{r, \mathbf{w}}_{h+kd(\mathcal{A})},\,\, k \in \mathbb{Z}
    \end{equation}
    is
    governed by the kernel $\bar p$, with boundary conditions
    $\bar {\mathbf{w}}^{(h)}
=( w_{h+kd(\mathcal{A})}, k \in \mathbb {Z})$, with
    ${r}^{(h)}=\max \{k: h+kd(\mathcal{A})\leq r\}$.
    If uniqueness holds
    for $p$, then $\mathbf {X}^{r, \mathbf{w}}$ converges weakly as
    $r \downarrow -\infty$ to the unique element $\mu$ of $\mathcal
    {G}(p)$, for any choice of $\mathbf{w} \in G^{\mathbb{Z}}$. Let
    $\mathbf {Y}$ be a process with distribution $\mu$. Likewise the process
    $\bar {\mathbf{X}}_{(h)}$
    converges weakly to $\mathbf {Y}^{(h)}=(Y_{h+kd(\mathcal{A})}, k \in
    \mathbb{Z})$, as ${r}^{(h)} \downarrow -\infty$
    which proves that uniqueness holds also for $\bar p$.

    For the converse notice that, for any $r$, conditionally to $\mathbf{w} \in
    G^{\mathbb{Z}}$, the $d(\mathcal {A})$-marginal processes $\bar {\mathbf {X}}_{(h)}$,
    defined in \eqref{marginal}, are independent, for $h=0,1,\ldots,d(\mathcal
    {A})-1$. By consequence, if $\bar \mu$ is the unique element in $\mathcal
    {G}(\bar {p})$, each of the $d(\mathcal {A})$-marginal processes converges to it, so the
    whole process $\mathbf {X}^{r, \mathbf{w}}$ has a limit
    distribution with the $d(\mathcal {A})$-marginal processes
 $\bar \mu$ distributed and independent, which ends the proof.

\end{proof}

The previous proposition corrects the erroneous statement contained
in our paper \cite{DPautor} (see Proposition 1 and Theorem 1) that
$d(\mathcal {A})=1$ is necessary for uniqueness.

Next assume that $d(\mathcal {A})=1$ and define the period of $i \in G$ to be
$d_i=gcd(s(\mathcal {A}^*_i))$, where $\mathcal {A}^*_i=\{\mathbf
{a} \in \mathcal {A}^*: P_{\mathbf {a}}(i,i)>0\}$. If $d_i=1$ we say
that the state $i \in G$ is \emph{aperiodic} for $P_{(\cdot)}$, The
following proposition guarantees that, for irreducible
$G$-stochastic functions, we can refer the term to the whole
function, since all states have the same period.

\begin{proposition}\label{period}
Let $P_{(\cdot)}$ be an irreducible $G$-stochastic function defined
on $\mathcal {A}$, with $d(\mathcal {A})=1$. Then $d_i=\hat d$, for
$i \in G$, for some $\hat{d} \in \mathbb{N}_+$. Moreover, there
exists a partition of $G$ in sets $\{G_h, h=0,\ldots ,  \hat d -1\}$ such
that
$$P_{\mathbf {b}}(i,j)>0, i \in G_h \Rightarrow j \in G_{h+s(\mathbf {b})},$$
identifying $G_{h+k\hat d}$ with $G_h$.
\end{proposition}

\begin{proof} By irreducibility for any pair $i,j \in G$
there exists $\mathbf{a}_1,\mathbf{a}_2 \in \mathcal{A}^*$ such that
$P_{\mathbf{a}_1}(i,j)P_{\mathbf{a}_2}(j,i)>0$. This implies that
$P_{\mathbf{a}_2\mathbf{a}_1}(i,i)\geq
P_{\mathbf{a}_1}(i,j)P_{\mathbf{a}_2}(j,i)>0$, so
$s(\mathbf{a}_2\mathbf{a}_1)=s(\mathbf{a}_1)+s(\mathbf{a}_2)=kd_i$.
Moreover if $P_{\mathbf{b}}(i,i)>0$, then
$P_{\mathbf{a}_1\mathbf{b}\mathbf{a}_2}(j,j)\geq
P_{\mathbf{a}_2}(j,i)P_{\mathbf{b}}(i,i)P_{\mathbf{a}_1}(i,j)>0$,
thus $\mathbf{a}_1\mathcal{A}^*_i\mathbf{a}_2 \subset
\mathcal{A}^*_j$, from which
$$
s(\mathbf{a}_1)+s(\mathcal{A}^*_i)+s(\mathbf{a}_2)=kd_i+s(\mathcal{A}^*_i)\subset
s(\mathcal{A}^*_j).
$$
This implies that $d_j\leq d_i$. Exchanging the
roles between $i$ and $j$ one gets $d_i=d_j$ as promised.

For the second statement let us fix some reference state $k \in G$,
and, for any $h=0,1,\ldots,{ \hat d}-1$ define the subsets of $G$
$$
G_h=\{i \in G: P_{\mathbf {a}}(k,i)>0 \text{ for some } \mathbf {a}
\in \mathcal {A}^* \text { with } s(\mathbf {a})=n\hat
{d}+h, n \in \mathbb{N}\}.
$$
By irreducibility the union of the $G_h$'s is the whole $G$. Now
suppose that $i \in G_{h_1}\cap G_{h_2}$. Then there exist $\mathbf
{a}_1 \in \mathcal {A}^*$ with $s(\mathbf{a}_1)=n_1\hat {d}+h_1$ and
$\mathbf {a}_2 \in \mathcal {A}^*$ with $s(\mathbf {a}_2)=n_2\hat
{d}+h_2$ such that $P_{\mathbf {a}_1}(k,i)P_{\mathbf {a}_2}(k,i)>0$.
We can safely assume that $n_1=n_2=n$, since $s (\mathcal{A}^*_k)$
contains all the multiples  of $\hat d$ large enough. Next let
$\mathbf {b} \in \mathcal {A}^*$ such that $P_{\mathbf {b}}
(i,k)>0$: we will have that
$$
P_{\mathbf {b}\mathbf {a}_1}(k,k) P_{\mathbf {b}\mathbf
{a}_2}(k,k)\geq P_{\mathbf {b}} (i,k)^2P_{\mathbf
{a}_1}(k,i)P_{\mathbf {a}_2}(k,i)>0,
$$
which implies that $s (\mathbf {b}\mathbf {a}_1) $ and  $s(\mathbf
{b}\mathbf {a}_2)$ are multiples of $ \hat d$. Hence
$$
s(\mathbf {b}\mathbf {a}_1)-s(\mathbf {b}\mathbf {a}_2)=s(\mathbf
{a}_1)-s(\mathbf {a}_2)=h_1-h_2
$$
must be a multiple of $\hat {d}$. Since $|h_2-h_1|<\hat d$ this
happens only when $h_1=h_2$. So we have proved that $\{G_h,
h=0,1,\ldots, \hat {d}-1\}$ is a partition.

Next  assume $P_{\mathbf {b}}(i,j)>0, i \in G_h, j \in G_l$. By
assumption there exists $\mathbf {a} \in \mathcal {A}^*$ such that
$P_{\mathbf {a}}(k,i)>0$ with $s(\mathbf {a})=n\hat d+h$. Then
$P_{\mathbf {b}\mathbf {a}}(k,j)\geq P_{\mathbf {a}}(k,i)P_{\mathbf
{b}}(i,j)>0$, so $$ s(\mathbf {b}\mathbf {a})=s(\mathbf
{b})+s(\mathbf {a})=s(\mathbf {b})+n\hat d+h=m\hat d +l
$$
for some integer $m$, from which
$$
s(\mathbf {b})=(m-n)\hat d+(l-h) \Rightarrow l=h+ {s(\mathbf {b})},
\text{ mod } \hat d,
$$
as desired.
\end{proof}

\medskip

We can immediately make profit of the previous proposition to establish
the following result.

\begin{theorem}\label{t2} Let $p$ be a transition kernel of the form \eqref{kern} and let $d(\mathcal {A})=1$.
If $P_{(\cdot)}$ is irreducible but not aperiodic (i.e. $ \hat d
>1$) then $|\mathcal{G}(p)|>1$.
\end{theorem}
\begin{proof} By Proposition \ref{period} the classes $G_0 , G_1,
\ldots , G_{\hat d -1}$ are well defined. Let us select an element from
each class, say $ g_i \in G_i$, for $i =0, \ldots , \hat d -1$ and
define $ w_k = g_i $ if $k $ is congruent to $i$ $mod (\hat d)$, for
$k \in \mathbb{Z}$. Finally define
 $\mathbf{w}= (w_k: k \in \mathbb{Z}) $ and the translated
 $ \hat {\mathbf{w}}$, with $\hat w_k= w_{k+1}$, for $k \in \mathbb{Z}$.

Recall that any probability measure in  $\mathcal{G}(p)$  can be obtained as a
weak limit of the laws of $\mathbf {X}^{r, \mathbf{w}}$, for $r
\downarrow -\infty$. These laws give probability one to the event $
\{ X_0 \in G_0\}$, so this remains true in the limit. On the other hand the weak limits of the laws of
$\mathbf {X}^{r, \hat {\mathbf{w}}}$ give probability one to the
event $ \{ X_0 \in G_1\}$, therefore the measure are necessarily
distinct. This ends the proof.
\end{proof}

An example of application of Theorem \ref{t2} is the situation
examined in \cite{DPautor}. If the image of $\mathcal {A}$ under
$P_{(\cdot)}$ is made by the matrices $I_2$ and $J_2$, then
$P_{(\cdot)}$ is irreducible. This happens also if this image is the
singleton $\{J_2\}$. Now suppose $d(\mathcal {A})=1$ and, for $k \in
\mathcal {A}$ and odd, $P_{(k)}=J_2$ and for $k \in \mathcal {A}$
and even, $P_{(k)}=I_2$. Then, for any state $i \in G$, $s(\mathcal
{A}^*_i)$ is made only by even numbers, thus $d_i$ is a multiple of
$2$. It is precisely $2$ since $d(\mathcal {A})=1$. Thus
$P_{(\cdot)}$ is not aperiodic. The two elements of $\mathcal
{G}(p)$ constructed in Theorem \ref{t2} are Dirac measures supported
by the two coherent sequences alternating the two states, as defined
in \cite{DPautor}.
Clearly these measures are not stationary: as a matter of fact the
convex combination of them with equal weights is the unique
stationary element in $\mathcal {G}(p)$. More generally, under the
hypotheses of Theorem \ref{t2}, we will prove in Section
\ref{perfect} that $\mathcal {G}(p)$ contains only one stationary
measure.

We close this section focusing our attention on the set $\mathcal
{I}$ of invariant distributions for the stochastic matrix $\hat
P=\sum_{k \in \mathcal {A}} \theta_k P_{(k)}$.
\begin{lemma}\label{ultimo} The following statements hold:
\begin{itemize}
\item[1.]Let $\lambda \in \mathcal {I}$ and $\mathbf {w}$  be a configuration with $\lambda$-distributed
single-site marginals. Then, for any $r \in \mathbb {Z}$, $\mathbf {X}^{r, \mathbf {w}}$ has the same property.
\item[2.] For any $\lambda \in \mathcal {I}$  there is at least one element of $\mathcal {G}(p)$
with single-site marginals equal to $\lambda$. This element can be always chosen to be stationary.
\item[3.]A stationary element of $\mathcal {G}(p)$ has all its single-site
marginals equal to an element of $\mathcal {I}$.
\item[4.] If there is a unique (stationary) element in $\mathcal {G}(p)$, there is a
unique invariant measure for $\hat P$.
\end{itemize}
\end{lemma}
\begin{proof}
For 1. it is enough to notice that, by induction on $n>r$, and interchanging the two sums of positive terms,
\begin{equation}\label{diciotto}
P(  X^{r, \mathbf{w}}_n =g  ) = E (P(  X^{r, \mathbf{w}}_n =g \, |\,
X^{r, \mathbf{w}}_i, \, i<n )) = \sum_{k \in \mathcal{A}} \theta_k
(\lambda P_{(k)}) (g)= (\lambda \hat P) (g) = \lambda
 (g) .
\end{equation}
For 2. take $\mathbf {X}^{r, \mathbf {w}}$  as above and send $r$ to $-\infty$: any
weak limit point will be in $\mathcal {G}(p)$ and it will keep the single-site marginals
equal to $\lambda$. By taking Cesaro averages on a window increasing to $\mathbb {Z}$ and
going to the limit one gets at least one stationary process with single site marginals still equal to $\lambda$.
For 3., let $\mathbf {X}\in \mathcal {G}(p)$ be stationary and call $\lambda$ its single-site
marginals: then, similarly to \eqref{diciotto}
\begin{equation}
\lambda(g)=P(  X_n =g  ) = E (P(  X_n =g \, |\,
X_i, \, i<n )) = \sum_{k \in \mathcal{A}} \theta_k
(\lambda P_{(k)}) (g)= ( \lambda \hat P) (g).
\end{equation}
4. is an immediate consequence of 3.
\end{proof}

By the remark following Proposition \ref{decompose}, it is
immediately seen that $\mathcal {I}$ has a single element if and
only if $P_{(\cdot)}$ is essentially irreducible. So this is a
necessary condition for the existence of a unique stationary element
in $\mathcal {G}(p)$. At the end of Section \ref{perfect} we will
able to prove that this condition is also sufficient.

\section{Main results}\label{princris}

After the results of the previous section, it remains to consider
transition kernels of the form \eqref{kern} with a corresponding
$G$-stochastic function $P_{(\cdot)}$ which is irreducible and
aperiodic with $d ( \mathcal{A})=1$. In this section we prove that
uniqueness holds for all of them. Here is the main result.

\begin{theorem}\label{t1}
Let $p$ have the form \eqref{kern} and assume that $P_{(\cdot)}$ is
irreducible and aperiodic. Then $ \mathcal{G}(p)$ has a unique
element $\mu$, with single-site marginals equal to $\hat {\lambda}$,
the unique invariant distribution of $\hat P$.
\end{theorem}

Before giving the proof of this result, we prove a weaker result in
order to present the basic argument in a simpler context. Recall
that under the above assumptions, $\hat P$ has a unique invariant
distribution.

\begin{lemma}\label{lemmatec}
Let $p$ have the form \eqref{kern} and let $P_{(\cdot)}$ be
irreducible and aperiodic. Then, for any $n \in\mathbb{Z}$ the
distribution of $X^{r, \mathbf{w}}_n$ converges to $\hat {\lambda}$,
as $r \to -\infty$, irrespectively of $\mathbf{w} \in
G^{\mathbb{Z}}$.
\end{lemma}
\begin{proof}
Without loss of generality, we take $d(\mathcal {A})=1$ and
$n=0$. We are going to couple in a suitable way the random variables
$X_0^{r,\mathbf{w}}$ for any $\mathbf{w} \in G^{\mathbb{Z}}$, in
such a way that, as $r \rightarrow -\infty$, they all converge to
the same random variable a.s.

Let $i_0$ be a fixed state in $G$. By irreducibility the additive
semigroup $s(\mathcal {A}_{i_0}^*)$ is non empty and, by
aperiodicity, it has a gcd equal to $1$: by the characterization of
these semigroups, there exists $m_0 \in \mathbb {N}_+$ such that,
for $m \geq m_0$, then $m \in s(\mathcal {A}_{i_0}^*)$. Now use
again irreducibility to prove that there exists $n_0 \in
\mathbb{N}_+$ such that, for any $i, j \in G $ one can choose a word
$\mathbf {b}_{i,j} \in \cal {A}^*$ with $P_{\mathbf
{b}_{i,j}}(i,j)>0$ and $s(\mathbf {b}_{i,j})=n_0$. We call
$\mathcal{B}$ the collection of these (distinct) words. These words
can be identified with paths in $\Gamma_{\theta}$ with the same
depth $n_0$. For any $\mathbf {b} \in \mathcal{B}$ we call
${\theta}_{\mathbf {b}}$ its probability and let $\rho=\sum_{\mathbf
{b} \in \mathcal{B}} {\theta}_{\mathbf {b}}$ be the sum of the
probabilities of these paths.

Next we are going to construct the random walk $\mathcal
{T}^{(0)}=\{T^{(0)}_k, k \in \mathbb{N}\}$ by generating
$\{K_{T^{(0)}_k}, k \in \mathbb{N}\}$, according to \eqref{rwalk}.
Let $(l_h^{(0)},u_h^{(0)}]$ be the interval between whose endpoints
a path corresponding to a word in $\mathcal {B}$ appears for the
$h$-th time in the random walk $\mathcal {T}^{(0)}$ (it is
understood that such a word can vary with $h$). By definition
$u_h^{(0)}-l_h^{(0)}=n_0$ (see Fig. 2). The independence of the
decrements of the random walk ensures that, with probability $1$, a
sub-walk in $\mathcal{B}$ will appear infinitely many times. Once
$\mathcal {T}^{(0)}$ reaches the site $V_r^{(0)}$, the value
$X_0^{r,\mathbf{w}}$ can be obtained from $w_{V_r^{(0)}}$ by using
the recursion in \eqref{back}, with an exception for each of the
sub-walks in $\mathcal {B}$ identified before. Suppose that one of
these sub-walks joins $s=T^{(0)}_k$ with $s-n_0=T^{(0)}_{k+m}$, with
$s-n_0>r$. Then, conditionally to this event, the transition from
$X_{s-n_0}^{{r,\mathbf {w}}}$ to $X_{s}^{{r,{\mathbf{w}}}}$ follows
the transition matrix
\begin{equation}\label{Q}
Q=\frac {1}{\rho}\sum_{\mathbf {b} \in \mathcal{B}}\theta_{\mathbf {b}}P_{\mathbf {b}}.
\end{equation}

\medskip

$$ \xymatrix{
\bullet & \bullet  \ar@/_1.5pc/[l]& \bullet & \bullet
\ar@/_1.5pc/[ll]  & \bullet & \bullet &\bullet \ar@/^1.5pc/[lll] &
\bullet & \bullet  \ar@/_1.5pc/[ll] &\bullet  \ar@/_1.5pc/[l] &
\bullet  & \bullet  \ar@/^1.5pc/[ll] } $$

\smallskip

\begin{center}\label{fig2}
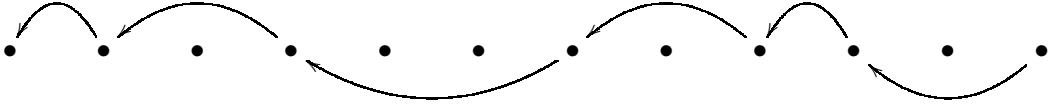
\small{Fig. 2. An example with $\mathcal {B}=\{(1,2),(2,1)\}$. Paths in $\mathcal {B}$ appear above the sites. }
\end{center}

\medskip

This matrix is positive by construction, therefore there exists a
coupling function $f_*:G \times [0,1] \rightarrow G$ with the
following property. If $U$ is uniform in $[0,1]$, $f_*(g,U)$ has the
law $Q(g,\cdot)$, and there exists $\epsilon>0$ such that for
$u<\epsilon$
\begin{equation}\label{coupfun}
f_*(g,u)=f_*(h,u), \forall g, h \in G  .
\end{equation}
It is enough to number the states and use the Skorohod
representation. Thus any transition from the left  endpoint to the
right one of the segments $(l_h^{(0)},u_h^{(0)}]$  can be performed
by drawing some independent random variable $U$ with uniform
distribution in $[0,1]$. When $U<\epsilon$ occurs, then
$X_{s}^{{r,{\mathbf{w}}}}$ does not depend on
$X_{s-n_0}^{{r,{\mathbf{w}}}}$, and thus does not depend on
$\mathbf{w}$. Once this \emph{coupling} occurs, by following the
recursion \eqref{back} one obtains that $X_{0}^{{r,{\mathbf{w}}}}$
does not depend on $\mathbf{w}$ as well.

Since each time a sub-walk in $\mathcal {B}$ occurs the $U$ used by
the coupling function $f_*$ are independent, coupling will happen
with probability $1$ and thus, as $r \rightarrow -\infty$,
$X_{0}^{{r,{\mathbf{w}}}}$ converges a.s. to a random variable which
does not depend of $\mathbf{w}$. Finally take $\mathbf{w}$ with all
the single-site marginal equal to $\hat \lambda$.
Then, by Lemma \ref{ultimo}, $X^{r, \mathbf{w}}$ has law $\hat \lambda$ for any $r$, and
this law is kept by any weak limit point. This ends the proof.

\end{proof}

\begin{proof}[Proof of Theorem \ref{t1}]
Again, for any finite subset $\Lambda
\subset \mathbb {Z}$, we have to couple all the processes $\mathbf
{X}_{\Lambda}^{r, \mathbf{w}}$ for all $r \in \mathbb {Z}$ and
$\mathbf {w} \in G^{\mathbb {Z}}$ in such a way that, as $r \to
-\infty$, they converge a.s. to the same limit vector, hence they
share the same limit in law.

Thus, we proceed first to the construction of the random walks
$\mathcal {T}^{(n)}=\{T^{(n)}_k, k \in \mathbb{N}\}$, for $n \in
\Lambda$, by using the i.i.d. $\theta$-distributed random variables
$(K_m, m \in \mathbb{Z})$. The intervals $(l_h^{(n)},
u_h^{(n)}]$ where a path in $\mathcal {B}$ occurs for the $h$-th
time, $h=1,2,\ldots$ are located within the random walk $\mathcal
{T}^{(n)}$, for any $n \in \Lambda$, with the constraint that
\emph{distinct} intervals which overlap are discarded.

We recall that the distance between two different random walks can
be seen as a von Schelling process until the possible coalescence.
The results in \cite{Boudiba} imply that when $d(\mathcal {A} )=1$
it is recurrent if and only if the distribution $\theta$ is
coalescent.
 In this case any finite family of random walks coalesce a.s.; otherwise
 this process is transient and any pair of distinct random walks
 $\mathcal{T}^{(m)}$ and $\mathcal{T}^{(n)}$ either coalesce or
 their distance goes to infinity a.s.

The set $\Lambda$ is partitioned according to the following
equivalence relation: $m \equiv n$ whenever $\mathcal {T}^{(m)}$ and
$\mathcal {T}^{(n)}$ coalesce. If this happens let $S_{m,n}$ denote
their coalescence point, otherwise we set $S_{m,n}= -\infty$. For
each of the equivalence classes $C_1,\dots,C_l$, let
\begin{equation}\label{coapo}
S_{C_h}= \left \{
\begin{array}{ll}
  \inf \{S_{m,n}, m, n \in C_h\}, & \text{if } |C_h| \geq 2 , \\
  n, & \text{if } C_h =\{n \},\\
\end{array} \right .
\end{equation}
for $  h=1,\ldots,l$. For a coalescent distribution $\theta$,
$l=1$ a.s.   Otherwise  the random walks $\mathcal
{T}^{(S_{C_h})}$ do not coalesce, for $h=1,\ldots,l$ and their
distance goes to infinity   with probability $1$.

 Therefore there are
infinitely many non overlapping intervals where a path in $\mathcal
{B}$ occurs, for each of these random walks. With the same argument
used in the previous proof we get that, with probability $1$,
provided $r$ is sufficiently close to $-\infty$,
$X_{S_{C_h}}^{{r,{\mathbf{w}}}}$ converge to $\hat {\lambda}$, irrespectively of $\mathbf {w}$,
independently for any $h=1,\ldots,l$.
Since any random
walk $\mathcal {T}^{(m)}$, for $m \in \Lambda$, necessarily hits the
set $\{S_{C_1},\ldots, S_{C_l}\}$, by forward iteration of
\eqref{back} it is then possible to obtain the required limiting
values for $X_{n}^{{r,{\mathbf{w}}}}$, for all $n \in \Lambda$.
\end{proof}

Next we apply the previous theorem to the class of binary autoregressive kernels
examined in \cite{DPautor}. Theorem 3 in this reference proved
uniqueness under conditions (a) and (b) therein. Now condition (a) is
nothing but irreducibility and aperiodicity of the associated
$G$-stochastic function. Condition (b) assumes coalescence so, as a consequence of
Theorem~\ref{t1}, this condition appears to be
unnecessary for uniqueness. Altogether
we have the following uniqueness theorem.
\begin{theorem}\label{vecchionuovo}
Consider a transition kernel $p$ of the form \eqref{kern} with
$|G|=2$, $d(\mathcal {A})=1$ and the range of $P_{(\cdot)}$
contained in $\{I_2, J_2\}$, with $I_2$ and $J_2$ defined in
\eqref{dimenticato}. Then uniqueness holds for $p$ if and only if
there exists $m \in \mathcal {A}$ even with $P_{(m)}=J_2$, or at
least $m_1, m_2 \in \mathcal {A}$, both odd, with $P_{(m_1)}=J_2$
and $P_{(m_2)}=I_2$.
\end{theorem}
\begin{proof}
It is clear that these conditions ensure that $J_2$ is contained in
the range of $P_{(\cdot)}$ (which is equivalent to irreducibility)
and that they exclude that the odd elements of $\mathcal {A}$ are sent by
$P_{(\cdot)}$ into $J_2$ and the even into $I_2$ (which is the only
periodic case).
\end{proof}

By putting together all the results proved so far, we finally get the following

\begin{theorem}\label{completo}
For a given kernel $p$ of the form \eqref{kern} one has uniqueness
if and only if the corresponding $G$-stochastic function
$P_{(\cdot)}$ is essentially irreducible and its single intercommunicating class is
aperiodic.
\end{theorem}

\section{Perfect  and $\varepsilon$-perfect
simulation}\label{perfect}

In this section, under the conditions of Theorem \ref{t1}  we
construct simulation algorithms for the unique element $\mu$ of
$\mathcal {G}(p)$ on a finite set of sites $\Lambda \subset \mathbb
{Z}$. Recall that we have assumed without loss of generality that
$d(\mathcal {A})=1$, so the coalescence property of the distribution
$\mathbf {\theta}$ is equivalent to the recurrence of the
corresponding von Schelling process.

The problem of determining conditions for coalescence of $\theta$
has been recently addressed in \cite{PW11}, where it has been
established that the following tail condition
$$
\sum_{k=1}^{\infty}(\sum_{n=k}^{\infty} \theta_n)^2 < + \infty
$$
is sufficient. Notice that this is weaker than the finiteness of the
mean of $\mathbf {\theta}$, which has the same form but without the
square. The latter is equivalent to the positive recurrence of the
corresponding von Schelling process, which is in turn equivalent to
the finiteness of the mean of the coalescence time of any two of the
random walks $\mathcal {T}^{(n)}, n \in \mathbb {Z}$. In \cite{PW11}
an example of transient von Schelling process was also given.

In the uniqueness regime, if $\mathbf {\theta}$ is coalescent a
perfect simulation algorithm for the marginal distribution
$\mu_{\Lambda}$ in any finite window $\Lambda\subset \mathbb {Z}$
can be designed. Indeed, in this case the coalescence point
$S_{m,n}$ is finite a.s. for any $n, m \in \Lambda$, and so is the
coalescence point
\begin{equation}\label{slambda}
S_{\Lambda}=\inf \{S_{m,n}, m \neq n \in \Lambda\} .
\end{equation}
Moreover these coalescence points are all adapted to the filtration
\begin{equation}\label{filtra}
\mathcal{F}^{\max \Lambda}_s=\sigma (K_n, s< n \leq \max \Lambda),
\,\, s< \min \Lambda,
\end{equation}
which makes them accessible through sequential simulation.
The following proposition essentially coincides with a result
appearing in \cite{DPautor} in a particular case.

\begin{theorem}\label{caso1}
Consider a transition kernel $p$ of the form \eqref{kern} with
$d(\mathcal {A})=1$, a coalescent distribution $\theta$, and
$P_{(\cdot)}$ irreducible and aperiodic. Let $\hat \lambda$ be the
unique invariant distribution of the stochastic matrix $\hat {P}$.
Then, for any finite $\Lambda\subset \mathbb {Z}$  the random vector
$\mathbf {X}_{\Lambda}=(X_n, n \in \Lambda)$ given by Simulation
Algorithm 1 is distributed as  $\mu_{\Lambda}$, the marginal on
$\Lambda$ of the unique element $\mu$ of $\mathcal {G}(p)$.

\medskip

{\bf Simulation Algorithm 1}

\begin{itemize}
    \item[1.] Construct the random walks $\mathcal{T}^{(n)}, n \in
    \Lambda$ up to their coalescence time $S_{\Lambda}$ (see Fig. 3);
    \item[2.] Sample $\tilde {X}_{S_{\Lambda}} \sim \hat {\lambda}$;
    \item[3.] Keeping the $K_m$ used in the first step and sampling $U_m$ i.i.d. uniform
    in $(0,1)$ for $S_{\Lambda}<m\leq \max {\Lambda}$, compute
    $$X_n=F_{S_{\Lambda},n}(    K_m , U_m , S_\Lambda < m \leq n;
    \tilde {X}_{S_{\Lambda}}    ), n \in \Lambda.$$
\end{itemize}
\end{theorem}

\medskip

$$\xymatrix{
\bullet & \bullet & \bullet \ar@/^1.5pc/[ll] & \bullet \ar@/_1.5pc/[lll]
& \bullet & \bullet & \bullet \ar@/^1.5pc/[llll]
& \bullet \ar@/_1.5pc/[llll] & \bullet \ar@/^1.5pc/[ll] & \bullet
\ar@/_1.5pc/[ll]
} $$

\smallskip

\begin{center}\label{fig3}
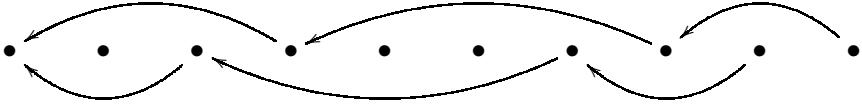
\small{Fig. 3. An example of coalescing random walks, with
$\mathcal {A}=\{2k, k \in \mathbb {N}_+\}\cup\{3\}$, for the simulation of the two rightmost
adjacent sites. The leftmost site is drawn from the distribution $\hat {\lambda}$.}
\end{center}

\medskip

\begin{proof} By assumption we are working in the uniqueness regime, so we can
approximate $\mu_{\Lambda}$ in total variation norm with the law of
$\mathbf {X}_{\Lambda}^{r,\mathbf{w}}$, with arbitrarily chosen
boundary condition $\mathbf{w}$, as $r \to -\infty$. Choosing the
components of $\mathbf{w}$ to be i.i.d. from the law $\hat
{\lambda}$ we can make profit of Lemma \ref{ultimo}. By the
strong Markov property applied to the random walks $\mathcal
{T}^{(n)}, n \in \Lambda$, conditionally to $ \{S_{\Lambda}=s \} $
with $s \geq r$, $X_s^{r,\mathbf{w}}$ is independent of $\mathcal
{F}_s^{\max \Lambda}$ and has the distribution of $\hat{\lambda}$.
By consequence we can represent the random variable $\tilde {X}$
produced in Step 2 of the algorithm as $X_s^{r,\mathbf{w}}$, which
implies that on the event
 $ \{S_{\Lambda}=s \}
$ with $s \geq r$,
$$
X_n^{r,\mathbf{w}}=F_{s,n}( K_m , U_m , s < m \leq n;X_s^{r,\mathbf{w}})=X_n, n \in \Lambda.
$$
By the coupling inequality (see e.g. \cite{Lind}) we have that the
total variation distance between the law of $\mathbf
{X}_{\Lambda}^{r,\mathbf{w}}$ and that of the output $\mathbf
{X}_{\Lambda}$ of the algorithm is bounded by the probability that
$S_{\Lambda}<r$. By sending $r$ to $-\infty$  the proof is
completed.
\end{proof}

Notice that if the unique invariant distribution $\hat {\lambda}$ for the
stochastic matrix $\hat P$ is difficult to compute, one can use a perfect
simulation algorithm for finite Markov chains to obtain a sample from it.

In the non-coalescent case the simulation algorithm of the previous theorem is unfeasible
since the coalescence point $S_{\Lambda}$ of $\Lambda$ is not finite with probability $1$.
Even if the distribution $\theta$ is coalescent we may need to stop the simulation
when it goes beyond some large negative threshold because of memory and time limitations.
A fortiori this needs to be done
if we are not able to prove coalescence. In all these cases it is still
possible to produce a sampling algorithm, provided a certain small error is accepted.
In order to evaluate this error
we need to introduce the following random time
\begin{equation}\label{ciao}    \hat S_{\Lambda}= \inf \{ S_{n,m}: n,m \in \Lambda, S_{n,m}>-\infty \}.
\end{equation}
In case the set appearing at the r.h.s. of \eqref{ciao} is empty we
define $\hat S_{\Lambda}=\min \Lambda$. As a consequence $\hat
S_{\Lambda}$ is finite, but when $\theta$ is not coalescent, it is
not adapted to the filtration \eqref{filtra}-

Here is a $\varepsilon$-approximate simulation algorithm, where
$\varepsilon$ is an error which goes to zero as the \
\emph{threshold site} $u \in \mathbb {Z}$ appearing in the algorithm
decreases to $-\infty$. Indeed, $\hat {S}_{\Lambda}$ being finite,
in principle it is possible to select $u$ sufficiently close to
$-\infty$ to make the r.h.s. of the forthcoming \eqref{stima}
smaller than any fixed $\varepsilon >0$. In practice, the
determination of a tail estimate on the distribution function of
$\hat {S}_{\Lambda}$ can be extremely complicate.

\begin{theorem}\label{epsilone}
Consider a transition kernel $p$ of the form \eqref{kern} with
$d(\mathcal{A}) =1$ and let $P_{(\cdot)}$ irreducible and aperiodic.
 Let $\hat \lambda$ be the
unique invariant distribution of the stochastic matrix $\hat {P}$
and $\mu$ be the unique element of $\mathcal {G}(p)$. Let
$\Lambda \subset \mathbb {Z}$ be finite, and let $u < \min \Lambda$.
Let $\tilde {\mu}^u_{\Lambda}$ be the law of the random vector
$\mathbf {\tilde X}^u_{\Lambda}=(\tilde X^u_n, n \in \Lambda)$
defined by the following

\medskip
{\bf Simulation Algorithm 2}

\begin{itemize}
    \item[1.] Construct the random walk $\mathcal{T}^{(n)}$ until $V^u_{(n)}$
    is reached, for $n \in \Lambda$;
    \item[2.] For any $m \in \mathcal{V}^{u,\Lambda}= \{V^u_{(n)} : n \in \Lambda\}$
    sample $\tilde {X}^u_{m} \sim \hat \lambda$, independently;
    \item[3.] Keeping the $K_m$ used in the  step 1
    and sampling $U_m$ i.i.d. uniform
    in $(0,1)$ for $u<m\leq \max {\Lambda}$, compute
    \begin{equation}\label{itera}
    {\tilde X}_n^{u}=F_{u,n}(K_m , U_m , u
    < m \leq n ;\tilde {X}^u_{V_u^{(n)}}),\,\,\, n \in \Lambda.
    \end{equation}
\end{itemize}

Then, in total variation norm
\begin{equation}\label{stima}
||\tilde {\mu}^u_{\Lambda}-{\mu}_{\Lambda}||\leq P(\hat
{S}_{\Lambda} <u) .
\end{equation}
\end{theorem}

\begin{proof}
Choosing the same boundary condition $\mathbf {w}$ as in the
previous theorem we can decompose the formula \eqref{F} in the
following two steps: first the boundary conditions are propagated up
to $\mathcal {V}^{u, \Lambda}$
\begin{equation}\label{coupdef1}
X_m^{r, \mathbf{w}}=F_{r,m}(K_l , U_l , r < l \leq m ;
w_{V^r_{(m)}}),\,\,\, m \in \mathcal {V}^{u, \Lambda} ,
\end{equation}
for any $u< \min{\Lambda}$.  Next  the random variables at $\Lambda
$ are constructed
\begin{equation}\label{coupdef3}
X_n^{r, \mathbf{w}}=F_{u,n}(K_l , U_l , u< l \leq n ; X^{r,
\mathbf{w}}_{V^u_{(n)}}),\,\,\, n \in \Lambda .
\end{equation}

Now we construct on the same probability space the random vector
$\tilde {\mathbf{X}}^u_{\mathcal {V}^{u, \Lambda}}$ produced by step 2 of
the algorithm. It will be shown that, when the
event $\{\hat{S}_{\Lambda} \geq u\}$ occurs,
 $\tilde {\mathbf{X}}^u_{\mathcal {V}^{u, \Lambda}}$
will coincide with $\mathbf{X}^{r, \mathbf{w}}_{\mathcal {V}^{u,
\Lambda}} $, as given by \eqref{coupdef1}. Therefore, on this event
also $\mathbf{X}_{\Lambda}^{r, \mathbf{w}}$, given by
\eqref{coupdef3}, will
   coincide with the output of the algorithm
   $\tilde {\mathbf {X}}_{\Lambda}^{u}$
    given by \eqref{itera}. By applying again the coupling inequality
    and sending
     $r \to -\infty$  the proof will be concluded.

For the last step we define
i.i.d. sequences $(K^{(m)}_l,U^{(m)}_l),r< l\leq u$, , independently
for  any $m \in \mathcal{V}^{u,\Lambda} $.
From each site $m$, by means of the $K^{(m)}_l$'s, independent
random walks are
started. As done in the case of two walks, we assume
that the rightmost of them is updated first.
The same $(K_l,U_l)$'s defined to
construct $\mathbf {X}^{r, \mathbf {w}}$ are used, except at coalescence sites,
where  additional independent copies are sampled.
Analogously to \eqref{indiceM}, let us define ${\tilde V}_r^{(m)}$ to be
the site where the random walk starting from $m$  lands under the  threshold $r$.
At this site we use the original
boundary condition appearing in $\mathbf {X}^{r, \mathbf {w}}$.
If more  random walks land at the same site, additional
independent copies are sampled as before.

These boundary conditions are propagated forward in time on each walk,  using the $U^{(m)}_l$'s,
until the starting points $m \in \mathcal{V}^{u,\Lambda}$
are reached. Following the notation of \eqref{F}, we have
\begin{equation}\label{esagera}
 {X}_{m}^{r,\mathbf{w}^{(m)}}=F_{r,m}(K^{(m)}_l,U^{(m)}_l,
r<l\leq m;w^{(m)}_{V_r^{(m)}}), \,\,\, m \in \mathcal{V}^{u,
\Lambda}
\end{equation}
By Lemma \ref{ultimo},
$ {X}_{m}^{r,\mathbf{w}^{(m)}} \sim \hat {\lambda}$, and they
are independent by construction, so they are identical in law
to the $ {\tilde X}_m^{u}$'s generated in
step 2 of the algorithm as promised.

Finally, observe that when $\hat {S}_{\Lambda}\geq u$, no site is
visited twice by any of these random walks and therefore, for any $m
\in \mathcal{V}^{u,\Lambda}$, the random variables $X_m^{r,
\mathbf{w}} $ defined by \eqref{coupdef1} coincide with $
{X}_{m}^{r,\mathbf{w}^{(m)}}$ in \eqref{esagera}. Using the coupling
inequality and sending $r$ to $-\infty$ concludes the proof.
\end{proof}

As remarked by one of the referees, the previous two algorithms work
also in case of non-uniqueness to construct any stationary element
of $\mathcal {G}(p)$. Which element is picked up depends on the
invariant distribution of $\hat P$ used in Step 2. In the
essentially irreducible case the matrix $\hat P$ is itself
essentially irreducible, so its unique invariant distribution is the
only possible choice.

\begin{corollary} Consider a transition kernel $p$ of the form \eqref{kern} with
$d(\mathcal{A}) =1$ and let $P_{(\cdot)}$ irreducible but \emph{not
aperiodic}. Let $\hat \lambda$ be the unique invariant distribution
of the stochastic matrix $\hat {P}$. Then $\mathcal {G}(p)$ has a
unique stationary element $\mu_s$. Let $\Lambda \subset \mathbb{Z}$
be finite. Simulation Algorithm 1 (for coalescent distributions
$\theta$) constructs a random vector $\tilde {\mathbf
{X}}_{\Lambda}$ distributed as $\mu_{s,\Lambda}$. Simulation
Algorithm 2 constructs a random vector $\tilde {\mathbf
{X}}^u_{\Lambda}$ converging to $\mu_{s,\Lambda}$ as $u \to
-\infty$.
\end{corollary}
\begin{proof}
Consider any stationary element $\tilde \mu$ of $\mathcal {G}(p)$.
By Lemma \ref{ultimo} it has necessarily $\hat \lambda$-distributed
single-site marginals. Take the boundary condition $\mathbf {w}$
appearing in the proof of Theorem \ref{caso1} and \ref{epsilone} to
be $\tilde \mu$-distributed. Then $\mathbf {X}^{r, \mathbf {w}}$ has
the distribution ${\tilde \mu}_{\Lambda}$. Since this is either the
law or it is arbitrarily close in total variation to the random
vectors constructed by these algorithms, which do not depend on the
choice of $\tilde \mu$, one establishes both the statements of this
corollary.
\end{proof}
\comment{
\begin{rem}
By comparing the proofs of Theorem \ref{epsilone} and Theorem
\ref{t1} one observes that the equivalence relation $\equiv$ in
$\Lambda$ appearing in the latter, which requires the knowledge of
an infinite sequence of decrements, is approximated in the former by
$\equiv_u$ where $m \equiv_u n$ if and only if
$V_u^{(m)}=V_u^{(n)}$. Clearly  the relation $m \equiv_u n $ implies
$m \equiv n$.
\end{rem}
}

\section{A result with $G$ countable}\label{basta}

In this section $G$ is allowed to be countable. In this case the complete characterization
presented in Theorem \ref{completo} fails, despite the fact that the
conditions of essential irreducibility and aperiodicity continue to
make sense. Indeed, by the lack of compactness of the sets of
probability measures over $G^{\mathbb {Z}}$, existence is not
guaranteed. On one side this prevents the construction of more than
one compatible law when essential irreducibility and/or aperiodicity
do not hold, and on the other requires to strengthen these
assumptions to prove existence and uniqueness. In this section we
are going to provide an assumption of Doeblin type that, in case $G$
is countable, allows to prove existence and uniqueness of compatible
laws. For definiteness assume that either $G=\mathbb {N}_+$ or
$G=\{1,\ldots,|G|\}$.

\medskip
{\bf Hypothesis D.} For a kernel $p$ of the form \eqref{kern}, there exists a certain
state, say $1$ without loss of generality, and an integer $\bar n_0 \in \mathbb{N}_+$,
with the following property
\begin{equation}\label{infinito}
 \forall i \in G, \,\,\, \exists \mathbf {b}_i \in
 \mathcal{A}^* \text{ with } s(\mathbf {b}_i)=
 \bar n_0 \text{ such that } \inf_{i \in G} P_{\mathbf {b}_i} (i, 1)=:\varepsilon>0.
\end{equation}

Whereas in the countable case this assumption is strictly stronger
than essential irreducibility and aperiodicity, if $G$ is finite it
is actually equivalent for the following reason. First, the fact
that for any $i \in G$ one has $P_{\mathbf {b}_i} (i, 1)>0$ for some
$\mathbf {b}_i \in \mathcal {A}^*$ it implies that two different
irreducible classes cannot exist. Second, the fact that these words
can be chosen with the same depth denies the existence of a
non-trivial partition in periodic classes as in Proposition
\ref{period}.

Before stating the result,  define $\bar {\mathcal{B}}=\{\mathbf {b}_i, i \in G\}$,
$\bar {\rho}=\sum_{\mathbf {b} \in \bar {\mathcal{B}}}
\theta_{\mathbf {b}}>0$ and
\begin{equation*}
\bar Q=\frac {1}{\bar \rho} \sum_{\mathbf b\in  \bar{ \mathcal{B}}}
\theta_{\mathbf {b}}P_{\mathbf {b}}.
\end{equation*}
These quantities will play in the forthcoming result the same role as
played by $\mathcal {B}$, $\rho$ and $Q$ in the proof of Lemma \ref{lemmatec}.
Under Hypothesis D, this matrix has all the entries of the first column not smaller
than $\varepsilon$, so the Skorohod construction  gives a coupling
function $f_*:G \times [0,1] \rightarrow G$ that, in addition to the property that $f_*(g,U)$ has the law
$Q(g,\cdot)$ when $U$ is uniform in $(0,1)$, for any $g \in G$, satisfies
\begin{equation}\label{coupfun2}
f_*(g,u)=1, g \in G, 0<u<\varepsilon.
\end{equation}

We warn the reader that this coupling function enters explicitly in
Simulation Algorithm 3 presented below, differently from what
happened in the algorithms of the previous section. If the
cardinality of $G$ is large or infinite the computation of $\bar Q
$, and consequently of the coupling function $f^*$, can give rise to
accuracy and computational time problems. This is the reason for
which, even if the following theorem applies to $G$ finite as well,
and as such it provides a perfect simulation algorithm under the
uniqueness regime which is free of error, in practice it may be
preferable to accept a small error introducing a truncation
threshold in the simpler Simulation Algorithm 2, which, in addition
to the $(K_l,U_l)$'s, requires only the computation of the invariant
distribution $\hat \lambda$. Moreover Simulation Algorithm 3 does
not apply in a situation of non uniqueness, as in the non aperiodic
case.

\begin{theorem}\label{end}
Under Hypothesis D, there exists a unique element in
$\mathcal{G}(p)$, whose restriction to any finite $\Lambda \subset
\mathbb{Z}$ is the law of the random vector $\mathbf {X}_{\Lambda}$
given by Simulation Algorithm 3.

\medskip
{\bf Simulation Algorithm 3}
\begin{itemize}
    \item[1.] Set $\mathcal {V}=\Lambda$;
    \item[2.] If $\mathcal {V}=\emptyset$ stop.
    \item[3.] Otherwise set $m=\max \mathcal {V}$;
    \item[4.] Construct the random walk $\mathcal {T}^{(m)}$ and check if a
    subpath corresponding to a word in $\bar{ \mathcal{B}}$ appears
before it lands on or below a site $s$  in $\mathcal {V}$
(this is possible only if the distance between $m$ and $\mathcal V \setminus \{m\}$ exceeds $\bar n_0$);
    \item[5.] If such a subpath appears connecting, say,
    the sites $u$ and $l$ (with $u-l=n_0$), extract  $U^*$ uniform in $(0,1)$,
    independent of all the variables generated previously: if  $U^*<\varepsilon$, set $X_u=1$,
    compute by forward simulation $X_m$, using \eqref{back} and the coupling
    functions $f^*(U^*;\cdot)$ on the previously located segments $(l,u]$,
    together with $X_n, n \in \Lambda$ for all $n$ such that $\mathcal {T}^{(n)}$
    hits $m$, delete $m$ from $\mathcal {V}$ and go to 2.;
    \item[6.] Otherwise replace $m$ by $s$ in $\mathcal V$ (if $s$ is already
    in $\mathcal {V}$ this means that the two random walks have coalesced) and go to 2.
\end{itemize}
\end{theorem}
\begin{proof}
The main to prove is that the above algorithm stops in a random but a.s. finite time.
We already know that any two walks $\mathcal{T}^{(m)}$ and $\mathcal{T}^{(n)}$, with $m \neq n$,
either coalesce, or their distance go to infinity a.s. As a consequence it
happens with probability zero that the rightmost walk, constructed from site
$m$ at step 3, only finitely many times is at distance larger than $\bar n_0$ from
the current position of the other "surviving" walks, stored in the set
$\mathcal {V}$. Each time it is at a distance larger than $\bar n_0$ there is a probability
$\bar \rho \varepsilon$ to stop the walk with the determination of
the value $X_u$ at some site $u$, independently of the outcome of all
previous simulations. Hence a finite number of walks will all be stopped in a
finite time a.s. We call $W$ the leftmost site where a random walk is stopped.

Next consider any process of the form $\mathbf {X}^{r, \mathbf {w}}$
constructed through the recursion \eqref{F}. If this rule is modified by using on each
interval $(l,u]$ where a path in $\bar {\mathcal {B}}$ has occurred the coupling
function $f^*$, such process can be represented on the same probability space
where the algorithm is constructed, without changing its law.
Again by the coupling inequality the variation distance between $\mathbf {X}_{\Lambda}$ and
$\mathbf {X}^{r, \mathbf {w}}_{\Lambda}$ will be bounded by the probability that $\{W\leq r\}$.
The proof is completed by sending $r$ to $-\infty$.
\end{proof}


\end{document}